\documentclass[11pt,a4paper]{amsart}

\usepackage[a4paper,margin=1in]{geometry}

\usepackage{amsmath,amssymb,amsthm,mathtools}
\usepackage{stmaryrd,mathrsfs}
\usepackage{microtype}
\usepackage[backref=page, colorlinks = true, linkcolor = black, citecolor = blue]{hyperref}
\usepackage[nameinlink,noabbrev]{cleveref}
\usepackage{enumerate}
\usepackage{appendix}
\usepackage{comment}
\usepackage{xcolor}
\usepackage{graphicx}

\allowdisplaybreaks

\newcommand{\R}{\mathbb R}
\newcommand{\C}{\mathbb C}
\newcommand{\Sph}{\mathbb S}
\newcommand{\Z}{\mathbb Z}
\newcommand{\M}{\mathbf M}
\newcommand{\spt}{\operatorname{spt}}
\newcommand{\dist}{\operatorname{dist}}
\newcommand{\dd}{\,d}
\newcommand{\bracket}[1]{\llbracket #1\rrbracket_2}

\DeclareMathOperator{\Imm}{Im}

\newcommand{\PhiMap}{\Phi}

\newcommand{\tCone}{\widetilde{C}}
\newcommand{\mres}{\mathbin{\vrule height 1.6ex depth 0pt width
0.13ex\vrule height 0.13ex depth 0pt width 1.3ex}}

\newtheorem{theorem}{Theorem}[section]
\newtheorem{lemma}[theorem]{Lemma}
\newtheorem{proposition}[theorem]{Proposition}
\newtheorem{conjecture}[theorem]{Conjecture}

\theoremstyle{remark}
\newtheorem{remark}[theorem]{Remark}

\title[White's cone]{White's cone over the M\"obius band is area-minimising}
\author[Marco A. M. Guaraco]{Marco A. M. Guaraco}
\address{Department of Mathematics, Huxley Building, South Kensington Campus, Imperial College London, London, SW7 2AZ, UK}
\email{guaraco@imperial.ac.uk}

\author[Davide Parise]{Davide Parise}
\address{Department of Mathematics, Zeeman Building, University of Warwick, Gibbet Hill Road, Coventry CV4 7AL, UK}
\email{Davide.Parise@warwick.ac.uk}

\date{\today}
\subjclass[2020]{49Q05, 53A10, 49Q15}
\keywords{area-minimising cone, flat chains modulo two, M\"obius band, helicoid, slicing}

\begin{document}

\begin{abstract}
We prove that the cone $C(\Sigma)$ in $\R^4$  over a suitable minimal M\"obius band $\Sigma\subset \Sph^3$ is locally mass-minimising with coefficients in $\Z_2$. This proves a conjecture by Brian White in the context of models of boundary singularities of  Plateau's problem in ambient dimension $4$ or higher. The key insight is that the cylindrical slices $C(\Sigma)\cap \{(z,w)\in \R^2 \times \R^2 : |w|=r\}$ are minimal surfaces in a flat rank-one translation quotient of $\R^3$. In fact, they correspond to non-orientable minimising helicoids previously studied by Antonio Ros. We also show that, up to rigid motions, the minimising cones $C(\Sigma)\times \R^{n-4}\subset \R^n$ are the only non-planar minimal hypercones with boundary $\R^{n-2}$ which are smooth away from the boundary that satisfy the minimal cylindrical slice property.
\end{abstract}

\maketitle

\vspace{-10pt}

\section{Introduction}

A well-known result of Hardt--Simon establishes boundary regularity in all dimensions for mass-minimising hypersurfaces with integer coefficients \cite{Hardt-Simon}. However, in many applications it is more natural to work with currents having coefficients in $\Z_2$, in which case the situation is radically different. In his lecture notes on
geometric measure theory, Brian White points out that boundary singularities occur already in dimension $4$. His argument shows these singularities sometimes exist for topological reasons and cannot be avoided perturbatively.  He also conjectures that a specific cone in $\R^4$ is a model for this phenomenon; see \cite[Section~10, pp.~57--58]{WhiteChodosh2014}. We prove this conjecture in \Cref{thm:main}. 

\begin{conjecture}[White {\cite[Conjecture~10.5, p.~58]{WhiteChodosh2014}}]
The cone $C(\Sigma)$ over the minimal M\"obius band
\[
 \Sigma=\{(z,w)\in\Sph^3:z^2\overline w\in\R_{\geq0}\}
\]
is mass-minimising modulo two.
\end{conjecture}

The surface $\Sigma$ can also be obtained by joining points on two orthogonal great circles via geodesic segments. One end moves with speed one around the first circle (the boundary of the M\"obius band) and the other with speed two on the second circle (the centre of the band). Its stereographic projection is a shell-like embedding of a M\"obius band in $\R^3$ sometimes called the \emph{sudanese} M\"obius band, see Figure \ref{fig:sudanese-mobius}. 

\begin{figure}[htbp] 
\centering\includegraphics[width=350pt]{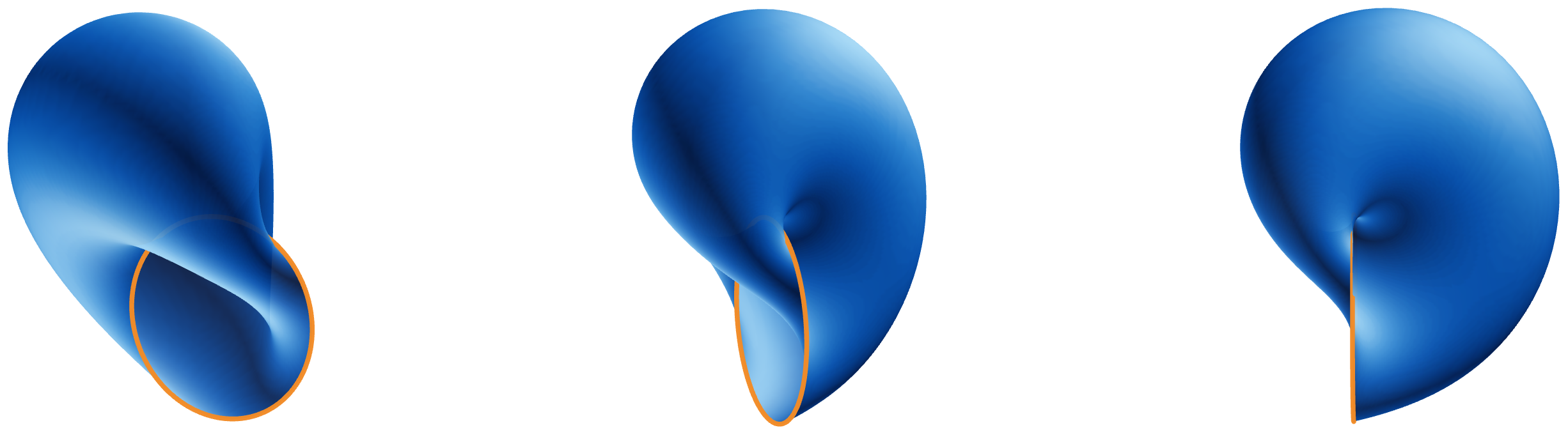}
    \caption{Three views of a stereographic projection of $\Sigma$.}
    \label{fig:sudanese-mobius}
\end{figure}

Local minimisation modulo two is preserved under products with
Euclidean spaces. Our result therefore shows that $C(\Sigma)\times\R^{n-4}\subset\R^n$ are area-minimising cones with boundary $\R^{n-2}$ for every $n\geq4$. Their singular set is $\{0\}\times\R^{n-4}$, so these examples attain the upper bound $n-4$ for the dimension of the boundary singular set of mass-minimising hypersurfaces modulo two; see \cite[Theorem~10.3]{WhiteChodosh2014}.

The cylindrical geometry used in the proof also leads to a characterisation of these cones. Under the regularity and connectedness assumptions of \Cref{thm:characterisation}, a minimal hypercone with linear boundary and a minimal cylindrical slice is either a half-hyperplane or a Euclidean product of White's cone. This characterisation requires only minimality of the cone and its slice. 

\subsection*{AI Declaration Statement} AI tools were used during the early stages of this project to help navigate the relevant literature and to carry out simple computations suggested and independently checked by the authors. At later stages, they were used to assist with improving the presentation of the manuscript and to check some standard definitions from geometric measure theory. The main ideas, arguments, and results of the paper are entirely the authors’ own.

\section{Organisation and main results}

Our main result follows via a slicing argument. We reduce the mass comparison for the cone against competitors to area comparisons in cylindrical slices, which turn out to be nonorientable helicoids in a flat Euclidean quotient. More precisely, let
\[
 C=C(\Sigma)=\{(z,w)\in\C^2:z^2\overline w\in\R_{\geq0}\},
 \qquad P=\{w=0\},
 \qquad \mathbf C_R=\bracket{C}\mres B_R.
\]
The cone $C$ is smooth away from the origin, where it has a singularity. As a chain modulo two, it has boundary $\partial\bracket{C}=\bracket{P}$. In \Cref{thm:main}, we show that for every $R>0$ and every compactly supported rectifiable flat $3$-chain $T$ with coefficients in $\Z_2$,
\[
 \partial T=\partial\mathbf C_R
 =\bracket{P\cap B_R}+\bracket{R\Sigma}
 \quad\Longrightarrow\quad
 \M(T)\geq\M(\mathbf C_R).
\]
This means $\bracket{C}$ is locally mass-minimising modulo two.

The key observation in the proof is regarding the level sets of the distance to the
boundary
\[
 p(z,w)=|w|.
\]
For each $\sigma>0$, the level-set slice $N_\sigma=p^{-1}(\sigma)$ is isometric to the flat product $\R^2\times(\R/2\pi\sigma\Z)$. Under this isometry, $C\cap N_\sigma$ is the quotient of a helicoid of pitch $2\sigma$ by the axial translation of length $2\pi\sigma$. This is precisely the nonorientable quotient
covered by Ros's area-minimisation theorem \cite[Theorem~13]{Ros2006}. Consequently, its compact subdomains minimise
mass modulo two in $N_\sigma$; the passage from smooth surfaces to (arbitrary) rectifiable flat-chain competitors is given in
\Cref{lem:Ros-flat-chain}. Finally, we slice a competitor $T$ by $p$ and integrate the resulting mass comparisons. Since $|\nabla_C p|=1$ almost everywhere on $C$, the coarea formula gives equality for the cone, while the slicing inequality gives the required lower bound for the competitor.

Our second main result is the characterisation in \Cref{thm:characterisation}. Let $n\geq4$, and let $T$ be a rectifiable flat $(n-1)$-chain modulo two of locally finite mass, with support $K$ and boundary $\partial T=\bracket{L}$, where $L\subset\R^n$ is an $(n-2)$-plane. Suppose that $K\setminus L$ is a smooth embedded minimal hypersurface and that $K$ is a smooth hypersurface with boundary $L$ near at least one point of $L$. If, for some $\sigma>0$, the entire slice
\[
 K\cap\{q\in\R^n:\dist(q,L)=\sigma\}
\]
is connected and minimal in the corresponding cylinder, then $K$ is either a half-hyperplane with boundary $L$ or, up to an orthogonal transformation,
\[
 K=C(\Sigma)\times\R^{n-4}.
\]
One positive radius suffices, since homotheties centred at the origin identify all positive cylindrical slices.

The proof first shows that the cone meets these cylinders orthogonally. On the Euclidean covering space of a cylinder, its slice is then foliated by hyperplanes whose normals rotate at constant speed. The existence of a regular boundary point forces a single half-turn, identifying the cone with a Euclidean product of White's cone.

The paper is organised as follows. Section~\ref{sec:cone} describes the geometry of the M\"obius band and its cone, establishes the boundary identity, and fixes our conventions for flat chains modulo two. Section~\ref{sec:helicoid} recalls Ros's theorem, formulates its flat-chain consequence, and identifies the cylindrical slices with quotient helicoids. Section~\ref{sec:proof} combines the slice comparisons with coarea to prove \Cref{thm:main}. Section~\ref{sec:characterisation} proves the characterisation of minimal cones with connected minimal cylindrical slices under the stated boundary regularity assumptions. Appendix~\ref{appendix:double-cover} records the algebraic relation between a double-cover lift of the cone and the Clifford torus, while Appendix~\ref{appendix:Ros-theorem} recalls Ros' arguments for the sake of completeness. 

\section{The M\"obius band and its cone}\label{sec:cone}

We identify $\R^4$ with $\C^2$ and write $(z,w)=(x_1+ix_2,x_3+ix_4).$ The construction suggested by White starts from the two orthogonal great circles
\[
    C_z:=\{(z,0):|z|=1\}=\{w=0\}\cap\Sph^3,
\]
and
\[
    C_w:=\{(0,w):|w|=1\}=\{z=0\}\cap\Sph^3.
\]
One lets a point move along $C_z$ with speed one and a point move along $C_w$ with speed two, and joins the corresponding points by the shortest geodesic segment in $\Sph^3$. Equivalently, setting
\[
    b(t)=(e^{it},0),\qquad c(t)=(0,e^{2it}),\qquad t\in\R/2\pi\Z,
\]
a convenient parametrisation of the resulting surface is
\begin{equation}\label{eq:f-link}
    f(s,t)=\cos s\,c(t)+\sin s\,b(t)
    =\bigl(\sin s\,e^{it},\cos s\,e^{2it}\bigr),
\end{equation}
where $ -\pi/2\leq s\leq \pi/2,$ and $0\leq t\leq\pi.$ Since $f(s,0)=f(-s,\pi),$ the quotient of the parameter rectangle under $(s,0)\sim(-s,\pi)$ is a M\"obius band. Its image $\Sigma\subset\Sph^3$ is smoothly embedded, its centre curve is
\[
    \{f(0,t):0\leq t\leq\pi\}=C_w,
\]
and its two parameter-boundary curves join to form the single great circle $ \partial\Sigma=C_z.$ Equivalently, the same image is obtained from
\[
    f(s,t)=\bigl(\sin s\,e^{it},\cos s\,e^{2it}\bigr),
    \qquad
    0\leq s\leq\frac{\pi}{2},
    \quad
    t\in\R/2\pi\Z;
\]
this second parametrisation is two-to-one only along the centre circle
$C_w$.

For completeness, set
\[
    h(s)^2:=\sin^2s+4\cos^2s=1+3\cos^2s.
\]
A direct computation gives
\[
    f^*g_{\Sph^3}=ds^2+h(s)^2\,dt^2,
    \qquad
    \nu=\frac{1}{h(s)}
    \bigl(2\cos s\,i e^{it},-\sin s\,i e^{2it}\bigr),
\]
where $\nu$ is a local unit normal to $\Sigma$ in $\Sph^3$. The metric is nondegenerate, including at $s=0$ and at the boundary. On the full parameter cylinder the free involution $\iota(s,t)=(-s,t+\pi)$ satisfies $f\circ\iota=f$ and $\nu\circ\iota=-\nu$. If $w\neq0$, equality of two images determines $t$ modulo $\pi$ and then the signed parameter $s$ modulo $\iota$; at $w=0$, the two edges are identified into a single embedded circle. Thus, $f$ induces an embedding of the M\"obius quotient, and $\nu$ does not descend to a global normal on it. With respect to this local normal, the second fundamental form satisfies
\[
    A_{ss}=A_{tt}=0,
    \qquad
    A_{st}=\frac{2}{h(s)}.
\]
It follows that
\[
    H_\Sigma=0,
    \qquad
    |A_\Sigma|^2=\frac{8}{h(s)^4},
\]
and hence $\Sigma$ is minimal. 
\begin{remark}
For a geometric verification of the minimality of $\Sigma$, fix the ruling $t=t_0$. The isometry
\[
 R_{t_0}(z,w)=\bigl(e^{2it_0}\overline z,e^{4it_0}\overline w\bigr)
\]
satisfies $R_{t_0}f(s,t)=f(s,2t_0-t)$, fixes the ruling great circle pointwise, and reverses the normal there. Equivariance of the mean-curvature vector therefore both fixes and negates that vector at each point of the ruling, so it must vanish.  
\end{remark}
We next record an implicit description of $\Sigma$. The real and imaginary parts of the complex-valued homogeneous polynomial $z^2\overline w$ are harmonic on $\R^4$, and on the parametrisation \eqref{eq:f-link} one has
\[
    z^2\overline w = \bigl(\sin s\,e^{it}\bigr)^2
    \bigl(\cos s\,e^{-2it}\bigr) = \sin^2s\cos s\in\R_{\geq0}.
\]
Conversely, suppose that $(z,w)\in\Sph^3$ and $z^2\overline w\in\R_{\geq0}$. If $z,w\neq0$, write
\[
    z=|z|e^{i\alpha},\qquad w=|w|e^{i\beta}.
\]
Then
\[
    z^2\overline w = |z|^2|w|e^{i(2\alpha-\beta)}
    \in\R_{\geq0}
\]
implies $2\alpha-\beta\in2\pi\Z.$ Since $|z|^2+|w|^2=1$, there is an $s\in[0,\pi/2]$ such that
\[
    |z|=\sin s,\qquad |w|=\cos s,
\]
and, taking $t=\alpha$ modulo $2\pi$,
\[
    (z,w)=\bigl(\sin s\,e^{it},\cos s\,e^{2it}\bigr).
\]
The cases $z=0$ and $w=0$ are similar. Therefore,
\begin{equation}\label{eq:implicit-link}
    \Sigma = \{(z,w)\in\Sph^3:z^2\overline w\in\R_{\geq0}\}.
\end{equation}
Let $C=C(\Sigma)$ denote the cone over $\Sigma$. From
\eqref{eq:implicit-link},
\begin{equation}\label{eq:cone-implicit}
    C = C(\Sigma) = \{(z,w)\in\C^2:z^2\overline w\in\R_{\geq0}\}.
\end{equation}
The direct conical parametrisation is
\begin{equation}\label{eq:cone-rst}
    F(r,s,t) = r f(s,t) = \bigl(r\sin s\,e^{it},r\cos s\,e^{2it}\bigr),
\end{equation}
where $r\geq0$, $s\in[0,\pi/2]$, and $t\in\R/2\pi\Z$. Introducing
\[
    \rho:=r\sin s=|z|,
    \qquad
    \sigma:=r\cos s=|w|,
\]
we obtain the equivalent parametrisation
\begin{equation}\label{eq:cone}
    C
    =
    \{(\rho e^{it},\sigma e^{2it}):
      \rho,\sigma\geq0,\ t\in\R/2\pi\Z\}.
\end{equation}
Conversely, for $r=\sqrt{\rho^2+\sigma^2}>0$,
\[
    r=\sqrt{\rho^2+\sigma^2},
    \qquad
    \sin s=\frac{\rho}{r},
    \qquad
    \cos s=\frac{\sigma}{r}.
\]
Thus the pair $(\rho,\sigma)$ simply replaces the cone radius $r$ and the transverse coordinate $s$; in particular, $r^2=\rho^2+\sigma^2.$ The parametrisation \eqref{eq:cone} is one-to-one away from $\{z=0\}$, where $t$ and $t+\pi$ determine the same point. The case $r=0$ gives the vertex. The cone $C$ is three-dimensional, whereas its link $\Sigma$ and the plane $P:=\{w=0\}$ are two-dimensional. The boundary statement is an identity of locally finite flat chains modulo two:
\begin{equation}\label{eq:boundary-cone}
    \partial\bracket{C}=\bracket{P}.
\end{equation}
Indeed, near $P\setminus\{0\}$ the map
\[
 (z,\sigma)\longmapsto
 \left(z,\sigma\frac{z^2}{|z|^2}\right),\qquad \sigma\geq0,
\]
gives one smooth boundary sheet. Near $z=0$, $w\neq0$, choose a local half-angle $\theta(w)$ of $w$; the parametrisation
$(a,w)\mapsto(ae^{i\theta(w)},w)$, with signed $a\in\R$, shows that this is a smooth interior locus, with no extra boundary. Finally, cutting out $B_\varepsilon$ adds the chain $\bracket{\varepsilon\Sigma}$ of mass $O(\varepsilon^2)$. Letting $\varepsilon\downarrow0$ proves \eqref{eq:boundary-cone} also at the vertex. All these chains have locally finite mass.

\subsection{Flat chains modulo two}\label{subsec:flat-chains}

We use locally flat chains with coefficients in $\Z_2$; see \cite[Sections~4.2--4.3]{Federer1969}. Write $\mathcal F_{k,\mathrm{loc}}(U;\Z_2)$ for the space of local flat chains and $\mathcal R_{k,\mathrm{loc}}(U;\Z_2)$ for its rectifiable chains of locally finite mass. The same notation will be used on smooth Riemannian manifolds, using charts. We write $\M(T)$, $\partial T$, $\spt T$ and $\|T\|$ for mass, boundary, support and mass measure. Total mass may be infinite, as it is for $\bracket{C}$ and $\bracket{P}$.

Orientations play no role and multiplicities take values in $\{0,1\}$. A properly embedded smooth submanifold $S\subset U$, possibly nonorientable and with boundary, determines its (canonical) chain $\bracket{S}$ when $S$ and $\partial S$ have locally finite measure. Then
\[
 \M(\bracket{S})=\mathcal H^k(S),\qquad
 \partial\bracket{S}=\bracket{\partial S}.
\]
Addition corresponds to symmetric difference up to $\mathcal H^k$-null sets; in particular, $T+T=0$.

For a locally Lipschitz map $F:U\to V$, the pushforward $F_\#T$ is used when $T$ has compact support, or when $F$ is proper on $\spt T$ and has the local mass control needed for pushforwards. It commutes with boundary operator. When $F$ is globally Lipschitz,
\[
 \M(F_\#T)\leq\operatorname{Lip}(F)^k\M(T).
\]
For a rectifiable chain $T$ of locally finite mass, the restriction $T\mres E$ is defined for Borel sets $E$, but its boundary need not have locally finite mass for an arbitrary $E$.

If $T$ and $\partial T$ have locally finite mass and $u:U\to\R$ is Lipschitz, then for almost every $t$ the slice
\[
 \langle T,u,t\rangle :=\partial(T\mres\{u<t\})+(\partial T)\mres\{u<t\}
\]
is a rectifiable local flat $(k-1)$-chain supported in $\spt T\cap\{u=t\}$. For finite-mass $T$, the slicing inequality and the boundary identity are
\begin{equation}\label{eq:slicing-standard}
 \int_\R\M(\langle T,u,t\rangle)\dd t
 \leq\operatorname{Lip}(u)\M(T),\qquad
 \partial\langle T,u,t\rangle=\langle\partial T,u,t\rangle.
\end{equation}
Note that there is no need to add a sign in front of the second identity as we are working modulo two. The corresponding local statements follow by restriction to relatively compact sets. If $T=\bracket{S}$ is smooth and $u|_S$ is a submersion along the level, then for almost every such $t$,
\begin{equation}\label{eq:slice-chain}
 \langle\bracket{S},u,t\rangle=\bracket{S\cap\{u=t\}},
\end{equation}
with multiplicity one.

A chain $T\in\mathcal R_{k,\mathrm{loc}}(U;\Z_2)$ is \emph{locally mass-minimising modulo two} if, for every $W\Subset U$ and every compactly supported finite-mass rectifiable cycle $Q$ with $\spt Q\Subset W$, one has
\[
 \M(T\mres W)\leq\M((T+Q)\mres W).
\]
Competitors of infinite mass give a trivial inequality. We distinguish this all-cycle comparison from \emph{minimisation in a fixed mod-two homology class}, which permits perturbations $Q=\partial V$. The latter is the property used for the closed surfaces in Appendix~\ref{appendix:Ros-theorem}; an arbitrary compact cycle there need not be a boundary.

\section{The quotient-helicoid slices}\label{sec:helicoid}

We first recall the result of Ros used below. For $a>0$, let
\[
 H_a=\{(s\cos t,s\sin t,at):s,t\in\R\}\subset\R^3
\]
be the helicoid of pitch $a$, and let
\[
 \tau_a(x_1,x_2,x_3)=(x_1,x_2,x_3+\pi a).
\]
Since $(s,t)$ and $(-s,t+\pi)$ have images differing by $\pi a e_3$, $H_a$ descends to a properly embedded one-sided surface $\overline H_a\subset\R^3/\langle\tau_a\rangle,$ for which we have the following result. 

\begin{theorem}[Ros]\label{thm:Ros}
For every $a>0$, the nonorientable quotient helicoid $\overline H_a$ is area-minimising modulo two. Equivalently, every smooth compact subdomain of $\overline H_a$ minimises area among compact orientable or nonorientable surfaces with the same boundary.
\end{theorem}

This is \cite[Theorem~13]{Ros2006}. For the sake of completeness, we have sketched in Appendix~\ref{appendix:Ros-theorem} Ros's proof, including the homological comparison that passes to the limit. For the main argument we use the following consequence of the smooth statement.

\begin{lemma}\label{lem:Ros-flat-chain}
Let $N_a=\R^3/\langle\tau_a\rangle$, and let $\overline H_a\subset N_a$ be the nonorientable quotient helicoid. Assume that every smooth compact subdomain $\Omega\Subset\overline H_a$ minimises area among all smooth compact surfaces, orientable or nonorientable, with the same boundary. Then, for every compactly supported rectifiable flat $2$-chain $S$ with coefficients in $\Z_2$ satisfying
\[
    \partial S=\partial\bracket{\Omega},
\]
one has
\[
    \M(S)\geq\M(\bracket{\Omega})=\mathcal H^2(\Omega).
\]
Consequently, $\bracket{\overline H_a}$ is locally mass-minimising modulo two.
\end{lemma}

\begin{proof}
Fix a smooth compact domain $\Omega\Subset\overline H_a$, set $\Gamma:=\partial\bracket{\Omega}$, and let $S$ be a compactly supported rectifiable chain modulo two with $\partial S=\Gamma$. We may assume $\M(S)<\infty$, otherwise the statement is trivial. Consider
\[
    m:=\inf\left\{\M(T):
    \begin{array}{l}
    T\in\mathcal R_{2,\mathrm{loc}}(N_a;\Z_2),\quad \M(T)<\infty,\\
    \spt T\text{ compact},\quad\partial T=\Gamma
    \end{array}\right\}.
\]
Since $S$ is admissible, $m\leq\M(S)$.

Writing $N_a\simeq\R^2\times(\R/\pi a\Z)$, choose $R$ large enough so that $\spt\Gamma\cup\operatorname{spt}S$ is contained in $K_R:=\overline{B_R^2}\times(\R/\pi a\Z)$. The nearest-point retraction
\[
    \Pi_R:N_a\rightarrow K_R
\]
is $1$-Lipschitz and fixes $\Gamma$. Hence the infimum defining $m$ may be taken over chains supported in $K_R$. A minimising sequence has uniformly bounded mass and the fixed (finite-mass) boundary $\Gamma$. Flat-chain compactness, rectifiability for finite coefficient groups, and lower semicontinuity therefore give a minimiser $T_*$; see \cite[Theorems~3.9 and~7.13]{WhiteChodosh2014} and \cite[Sections~4.2--4.3]{Federer1969}. Thus,
\[
    \partial T_*=\Gamma, \qquad \M(T_*)=m.
\]

The minimiser is locally mass-minimising in the whole flat manifold
$N_a$. Interior regularity in dimension two and codimension one gives no
interior singularities. Boundary regularity for a smooth prescribed
boundary gives a singular-set bound $m-3$, hence no boundary singularities
when $m=2$; see \cite[Theorems~9.13--9.14 and~10.3, p.~57]{WhiteChodosh2014}.
These statements apply in local Euclidean charts of $N_a$. Consequently,
\[
    T_*=\bracket{\Sigma_*}
\]
for a smooth compact embedded surface $\Sigma_*\subset N_a$, possibly
nonorientable, with $\partial\Sigma_*=\partial\Omega$. By the assumed
smooth minimising property of $\Omega$,
\[
    \M(\bracket{\Omega})
    =
    \mathcal H^2(\Omega)
    \leq
    \mathcal H^2(\Sigma_*)
    =
    \M(T_*)
    =
    m
    \leq
    \M(S).
\]
This proves the comparison. To obtain the local statement for all cycles, choose a smooth compact domain $\Omega\subset\overline H_a$ whose interior contains $\overline H_a\cap\spt Q$ for a given compactly supported cycle $Q$. Apply the comparison to $\bracket{\Omega}+Q$ in a relatively compact open set containing both supports, and cancel the identical mass outside $\Omega$. Conversely, the chain comparison applies to every smooth compact surface, whose (canonical) chain has mass equal to its area.
\end{proof}

Let now
\[
 p:\C^2\to[0,\infty),\qquad p(z,w)=|w|,
\]
and, for $\sigma>0$, set $N_\sigma=p^{-1}(\sigma)$. The map
\begin{equation}\label{eq:isometry-level}
 \Psi_\sigma:N_\sigma\rightarrow
 \R^2\times\bigl(\R/2\pi\sigma\Z\bigr),
 \qquad
 \Psi_\sigma(z,\sigma e^{i\beta})=(\Re z,\Im z,[\sigma\beta]),
\end{equation}
is a well-defined isometry for the metric induced from $\R^4$.

\begin{lemma}\label{lem:slice-helicoid}
For every $\sigma>0$, $C\cap N_\sigma$ is mapped by $\Psi_\sigma$ onto the quotient helicoid $\overline H_{2\sigma}$.
\end{lemma}

\begin{proof}
By \eqref{eq:cone},
\[
 C\cap N_\sigma =\{(se^{it},\sigma e^{2it}):s,t\in\R\}.
\]
Under \eqref{eq:isometry-level}, this becomes
\[
 \{(s\cos t,s\sin t,[2\sigma t]):s,t\in\R\},
\]
which is the image of $H_{2\sigma}$ in the quotient by the translation of length $2\pi\sigma=\pi(2\sigma)$.
\end{proof}

Fix $R>0$ and set
\[
 \mathbf C_R:=\bracket{C}\mres B_R.
\]
Its boundary is
\begin{equation}\label{eq:full-boundary}
 \partial\mathbf C_R =\bracket{P\cap B_R}+\bracket{R\Sigma}.
\end{equation}
The boundaries of the two displayed pieces cancel on the common circle $P\cap\partial B_R$. For $0<\sigma<R$, let
\[
 A_{\sigma,R}:=\sqrt{R^2-\sigma^2},
 \qquad
 C_{\sigma,R}:=C\cap\overline{B_R}\cap N_\sigma.
\]
Under $\Psi_\sigma$, the set $C_{\sigma,R}$ is a smooth compact subdomain of $\overline H_{2\sigma}$, namely the portion $|z|\leq A_{\sigma,R}$. Its boundary is the single closed curve
\begin{equation}\label{eq:gamma-sigma}
 \Gamma_{\sigma,R} =\{(A_{\sigma,R}e^{it},\sigma e^{2it}):t\in\R/2\pi\Z\}.
\end{equation}
Consequently, by \Cref{lem:Ros-flat-chain} and the isometry
$\Psi_\sigma$,
\begin{equation}\label{eq:slice-minimises}
 \M(S)\ge\mathcal H^2(C_{\sigma,R})
\end{equation}
for every compactly supported rectifiable flat $2$-chain $S$ modulo two in the slice $N_\sigma$ satisfying $\partial S=\bracket{\Gamma_{\sigma,R}}$.

\section{Mass-minimisation of the cone}\label{sec:proof}
We are now able to state and prove our main result. 
\begin{theorem}\label{thm:main}
The chain $\bracket{C}$ associated with the cone in \eqref{eq:cone} is locally mass-minimising modulo two. More precisely, for every $R>0$ and every compactly supported rectifiable flat $3$-chain $T$ modulo two such that
\[
 \partial T=\partial\mathbf C_R,
\]
one has
\[
 \M(T)\ge\M(\mathbf C_R).
\]
\end{theorem}

\begin{proof}
If $\M(T)=\infty$, the inequality is immediate. Assume $\M(T)<\infty$; its prescribed boundary \eqref{eq:full-boundary} also has finite mass. For almost every $\sigma>0$, define
\[
 T_\sigma:=\langle T,p,\sigma\rangle,
 \qquad
 \mathbf C_{R,\sigma}:=\langle\mathbf C_R,p,\sigma\rangle.
\]
By \eqref{eq:slicing-standard} and \eqref{eq:slice-chain},
\[
 \partial T_\sigma
 =\langle\partial T,p,\sigma\rangle
 =\langle\partial\mathbf C_R,p,\sigma\rangle
 =\partial\mathbf C_{R,\sigma}.
\]
For almost every $\sigma>0$, the rectifiable slice $T_\sigma$ is supported in the smooth manifold $N_\sigma$. Its approximate tangent planes lie in $TN_\sigma$, so it defines a rectifiable chain in $N_\sigma$ with the same mass and boundary. The same holds for $\mathbf C_{R,\sigma}$, and $\Psi_\sigma$ preserves both masses. For almost every $0<\sigma<R$, $\mathbf C_{R,\sigma}=\bracket{C_{\sigma,R}}$. Thus \Cref{lem:Ros-flat-chain}, in the form \eqref{eq:slice-minimises}, gives
\begin{equation}\label{eq:pointwise-slice}
 \M(T_\sigma)\ge\M(\mathbf C_{R,\sigma})
 \qquad\text{for a.e. }\sigma\in(0,R).
\end{equation}
No support condition $\spt T\subset\overline{B_R}$ is required: levels above $R$ only contribute nonnegative mass. The $0$ and $R$ level sets have zero measure. Since $p$ is $1$-Lipschitz, slicing gives
\begin{equation}\label{eq:competitor-coarea}
 \M(T)\ge\int_0^\infty\M(T_\sigma)\dd\sigma
 \ge\int_0^R\M(\mathbf C_{R,\sigma})\dd\sigma.
\end{equation}

It remains to identify the last integral. Away from a set of $\mathcal H^3$-measure zero, $C$ is parametrised by
\[
 F(\rho,\sigma,t)=(\rho e^{it},\sigma e^{2it}),
 \qquad \rho,\sigma>0,
\]
and
\[
 F_\rho=(e^{it},0),\qquad
 F_\sigma=(0,e^{2it}),\qquad
 F_t=(i\rho e^{it},2i\sigma e^{2it}).
\]
These vectors are mutually orthogonal and have lengths $1$, $1$, and $\sqrt{\rho^2+4\sigma^2}$, respectively. Thus
\begin{equation}\label{eq:cone-mass}
 \M(\mathbf C_R)
 =2\pi\int_0^R\int_0^{\sqrt{R^2-\sigma^2}}
 \sqrt{\rho^2+4\sigma^2}\dd\rho\dd\sigma.
\end{equation}
For a.e. $\sigma\in(0,R)$, the same parametrisation with $\sigma$ fixed gives
\begin{equation}\label{eq:slice-mass}
 \M(\mathbf C_{R,\sigma})
 =2\pi\int_0^{\sqrt{R^2-\sigma^2}}
 \sqrt{\rho^2+4\sigma^2}\dd\rho.
\end{equation}
Hence Fubini's theorem gives
\[
 \M(\mathbf C_R)=\int_0^R\M(\mathbf C_{R,\sigma})\dd\sigma.
\]
Combining this identity with \eqref{eq:competitor-coarea} proves the mass comparison when restricting to a ball of radius $R$. 

For local minimisation, let $Q$ be a finite-mass rectifiable cycle with $\spt Q\Subset W\Subset\R^4$. Choose $R$ with $\overline W\subset B_R$ and apply the comparison to $T=\mathbf C_R+Q$. On $\R^4\setminus W$ the two chains agree. Cancelling their identical mass there yields
\[
 \M(\bracket{C}\mres W)
 \leq\M((\bracket{C}+Q)\mres W),
\]
as required.
\end{proof}

\begin{remark}\label{rem:why-slicing}
The key feature of the function $p(z,w)=|w|$ is that its tangential gradient has unit length almost everywhere on $C$. Equivalently, the coarea inequality is an equality for the cone. At the same time, every positive-level slice is a locally mass-minimising chain modulo two by \Cref{lem:Ros-flat-chain}.
\end{remark}

\section{Characterisation by minimal cylindrical slices}
\label{sec:characterisation}

In this section only, we extend the notation $P$, $p$ and $N_\sigma$ from the preceding sections to higher ambient dimensions. For $n\geq4$, write
\[
 \R^n=\R^{n-2}_x\times\C_w,
 \qquad P=\{w=0\},
 \qquad p(x,w)=|w|,
\]
and set $N_\sigma=p^{-1}(\sigma)$ for $\sigma>0$. Thus, $\sigma$ still denotes distance to the boundary plane, and $N_\sigma$ carries its induced flat product metric. The symbols $C=C(\Sigma)$ and $\Sigma$ retain their original four-dimensional meanings.

A point of $P$ is a \emph{regular boundary point} of a cone $K$ if, in a neighbourhood of that point, $K$ is a smooth embedded hypersurface with boundary $P$. The following is our second main result. 

\begin{theorem}\label{thm:characterisation}
Let $T$ be a rectifiable flat $(n-1)$-chain modulo two of locally finite mass, with conical support $K$ and $\partial T=\bracket{P}$. Suppose that $K\setminus P$ is a smooth embedded minimal hypersurface and that $K$ has a regular boundary point. If, for some $\sigma>0$, the entire slice $K\cap N_\sigma$ is connected and minimal in $N_\sigma$, then either $K$ is a half-hyperplane with boundary $P$ or, up to an orthogonal transformation preserving $P$,
\[
 K=C(\Sigma)\times\R^{n-4}.
\]
Conversely, both types of cones satisfy the hypotheses and every positive cylindrical slice is minimal.
\end{theorem}

\begin{proof}
If $K$ is a half-hyperplane, the conclusion is immediate. Henceforth, assume that $K$ is not a half-hyperplane. Homotheties centred at the origin identify all positive cylindrical slices, so we work with $S=K\cap N_1$. By the hypotheses, $S$ is a connected, smooth, properly embedded hypersurface without boundary. Every point of $K\setminus P$ is a positive multiple of a unique point of $S$. We now divide the proof in five steps. 

\smallskip\noindent
\emph{Step 1. The cone meets the cylinders orthogonally.}
On $N_1$, define the vector field $X(x,w)=(x,0)$, and let
$Y=X^\top$ denote its tangential projection onto $S$. For a local unit normal $\nu$ to $S$ in $N_1$, set $u=\langle X,\nu\rangle$ and $q=\sqrt{1+u^2}$. The corresponding unit normal to the cone along $S$ is $(\nu-u(0,w))/q$. With these choices of normal, write $H_K$ for the scalar mean curvature of $K\setminus P$ in $\R^n$, and $H_S$ for that of $S$ as a hypersurface of $N_1$ with its induced metric. Parametrise $K\setminus P$ by
\[
 G:(0,\infty)\times S \rightarrow K\setminus P,
 \qquad G(\lambda,\xi)=\lambda\xi.
\]
Here $\xi\in S$ and $\lambda>0$ is the dilation factor about the origin; in particular, $G(1,\xi)=\xi$. Tracing the second fundamental form in this parametrisation at $\lambda=1$ gives
\begin{equation}\label{eq:characterisation-mean-curvature}
 qH_K=H_S+u-\frac{\langle Y,\nabla_S u\rangle}{1+u^2},
\end{equation}
Here $\nabla_S$ is the gradient on $S$, so that $\langle Y,\nabla_S u\rangle$ is the directional derivative of $u$ along $Y$. Both mean curvatures are traces, with the convention $A=-d\nu$. Indeed, writing $\eta=\langle\nu,(0,iw)\rangle$, the trace is $H_S+u(1-\eta^2)+(A_S(Y,Y)+u^3\eta^2)/(1+u^2)$, and $\langle Y,\nabla_S u\rangle=u\eta^2-A_S(Y,Y)$ gives the displayed identity.

Both mean curvatures vanish. Thus, the globally defined function
$b=u^2/(1+u^2)$ satisfies $\langle Y,\nabla_S b\rangle=2b$ and $0\leq b<1$. The flow of $Y$ is complete: along an integral curve
$\gamma(t)=(x(t),w(t))$, we have $\gamma'=Y$, so that 
\[
 \frac{d}{dt}|x|^2=2\langle x,x'\rangle =2\langle X,Y\rangle=2|Y|^2\leq2|x|^2.
\]
Here $\langle X,Y\rangle=|Y|^2$ because $Y=X^\top$, so $X-Y$ is orthogonal to $Y$. Thus $|x|$ stays bounded on finite time intervals, and properness of $S$ prevents escape. Since $\langle Y,\nabla_S b\rangle=2b$, on the same curve we have $b(\gamma(t))=e^{2t}b(\gamma(0))$. Completeness and the bound $0\leq b<1$ therefore force $b=0$. Hence $(x,0)$ is tangent to $S$; subtracting it from the cone's radial vector shows that $(0,w)$ is tangent to $K$. The cone therefore meets $N_\sigma$ orthogonally. The complete flow of $X$ also gives invariance under $(x,w)\mapsto(\lambda x,w)$; combining this with homotheties centred at the origin gives invariance under $(x,w)\mapsto(x,\lambda w)$.

\smallskip\noindent
\emph{Step 2. Sections by translates of $P$ are planes.}
Contraction in the $x$-factor sends each $(x,e^{i\beta})\in S$ to
$(0,e^{i\beta})\in S$. At such an axis point, the tangent space is invariant under $(v,t)\mapsto(\lambda v,t)$. It is therefore either horizontal or of the form $E\oplus\R\partial_\beta$, where
$E\subset\R^{n-2}$ has dimension $n-3$.

In the horizontal case, a local graph $\beta=h(x)$ satisfies $h(\lambda x)=h(x)$, so continuity at zero makes it constant. Dilation and connectedness then give $S=\R^{n-2}\times\{e^{i\beta_0}\}$, making $K$ a half-hyperplane, contrary to our assumption. Thus, at every axis point, we can write a local graph as $\zeta=h(\xi,\beta)$, with $\xi\in E$. Dilation invariance and differentiability at zero give
\[
 h(\lambda\xi,\beta)=\lambda h(\xi,\beta),
 \qquad h(\xi,\beta)=D_\xi h(0,\beta)\xi.
\]
Thus, nearby axis points belong to $S$, and the local sections are linear hyperplanes. The set of angles present on the axis is nonempty, open and closed, hence is the whole circle. Dilation extends each local section to its entire hyperplane. Consequently
$S=\{(x,e^{i\beta}):x\in E_\beta\}$ for a smooth periodic family of linear hyperplanes $E_\beta\subset\R^{n-2}$. In particular,
$K\cap\{w=\sigma e^{i\beta}\}=E_\beta\times\{\sigma e^{i\beta}\}$
is an affine $(n-3)$-plane in the corresponding translate of $P$.

\smallskip\noindent
\emph{Step 3. The planes rotate at constant speed.}
Unwrap the cylinder to its Euclidean covering space $\R^{n-2}\times\R$. Choose a smooth unit normal $\nu_0(\beta)$ to $E_\beta$ on this cover. The lifted slice is the level set $f(x,\beta)=\langle x,\nu_0(\beta)\rangle=0$, whose mean curvature is
\[
 H_S=-\frac{\langle x,\nu_0''(\beta)\rangle}
 {\bigl(1+\langle x,\nu_0'(\beta)\rangle^2\bigr)^{3/2}}.
\]
Since this vanishes for every $x\perp\nu_0$, we have
\[
 \nu_0''=-|\nu_0'|^2\nu_0,
 \qquad (|\nu_0'|^2)'=2\langle\nu_0',\nu_0''\rangle=0.
\]
The constant angular speed $\kappa:=|\nu_0'|$ is positive: otherwise $K$ would be an entire hyperplane. By the constancy theorem, $T$ is the canonical mod-two chain carried by $K$, so its boundary would be zero, contrary to $\partial T=\bracket{P}$. Thus $\nu_0''=-\kappa^2\nu_0$, and the normals rotate in the fixed two-plane $Q=\operatorname{span}\{\nu_0(0),\nu_0'(0)\}$. The common subspace
\[
 V=\bigcap_\beta E_\beta=Q^\perp\subset\R^{n-2}
\]
has dimension $n-4$, and $K$ is invariant under translations in $V$.

\smallskip\noindent
\emph{Step 4. There is exactly one half-turn.}
Periodicity of the unoriented planes gives $2\pi\kappa=k\pi$ for some integer $k\geq1$. Choose coordinates $x=(z,y)\in\C\times V$ so that
\[
 E_\beta=\{(z,y):\Imm(ze^{-ik\beta/2})=0\}.
\]
Regular boundary points form a nonempty relatively open subset of $P$, so choose one with $z\neq0$. Near it, write $z=|z|e^{i\alpha(z)}$. The cone consists of the $k$ distinct half-sheets
\[
 w=\sigma\exp\!\left(\frac{2i\alpha(z)+2\pi i j}{k}\right),
 \qquad \sigma\geq0,\quad j=0,\ldots,k-1.
\]
A regular boundary point has exactly one boundary sheet. Hence $k=1$.

\smallskip\noindent
\emph{Step 5. Identification of the cone.}
We have
\[
 K\setminus P =\{(s e^{i\beta/2},y,\sigma e^{i\beta}):
       s\in\R,\ y\in V,\ \sigma>0,\ \beta\in\R\}.
\]
Equivalently, $z^2\overline w=\sigma s^2\geq0$. Taking the closure gives $K=C(\Sigma)\times V$.

Conversely, a half-hyperplane with boundary $P$ satisfies the hypotheses and has totally geodesic cylindrical slices. The product $C(\Sigma)\times\R^{n-4}$ is minimal away from $P$, has mod-two boundary $P$, and is regular at every boundary point with
$z\neq0$. Its cylindrical slices are connected and have the form in Step~3 with $\nu_0(\beta)=(-\sin(\beta/2),\cos(\beta/2),0,\ldots,0)$; the displayed mean-curvature formula shows that they are minimal.
\end{proof}

\appendix
\numberwithin{equation}{section}

\section{The cone over the Clifford torus double covers $C(\Sigma)$}
\label{appendix:double-cover}

Although not needed for the proof of the main theorems, we record an interesting feature of the cone $C$. Let
\[
 X=\R^4\setminus P=\C\times\C^*.
\]
Since $\pi_1(X)=\Z$, there is a unique connected two-sheeted cover up to equivalence; the disconnected trivial cover is a different possibility. A representative of the connected cover, with domain again identified with $X$, is
\[
 \PhiMap:X \rightarrow X,\qquad \PhiMap(z,w)=(z,w^2).
\]
For $w=x+iy$, its real-coordinate expression is
$\PhiMap(z,(x,y))=(z,(x^2-y^2,2xy))$.

The metric making this map a local isometry onto the Euclidean base is
\begin{equation}\label{eq:cover-metric}
 g_\Phi=\PhiMap^*g_{\mathrm{Euc}}
 =|dz|^2+4|w|^2|dw|^2.
\end{equation}
This is not the Euclidean metric on the chosen coordinates on the cover. Arbitrary changes of the metric on the cover are not equivalent to changes of covering map over the fixed Euclidean base. In particular, a base dilation lifts to $(z,w)\mapsto(rz,\sqrt r\,w)$, and
\[
 \PhiMap^{-1}(\Sph^3\setminus P)=\{|z|^2+|w|^4=1,\ w\neq0\}.
\]
The round-sphere observation below is therefore algebraic; it does not allow to transfer minimisation, stability or uniqueness between the two metrics.

\begin{proposition}[The lift and its closure]\label{prop:Clifford-closure}
The lift of $C$ and its closure in $\R^4$ are
\begin{equation}\label{eq:tildeC}
 \begin{aligned}
 \tCone&:=\PhiMap^{-1}(C\cap X)
 =\{(z,w):\Im(z\overline w)=0,\ w\neq0\},\\
 Q&:=\overline{\tCone}^{\,\R^4}=\{\Im(z\overline w)=0\}.
 \end{aligned}
\end{equation}
The set $\tCone$ is dilation-invariant but does not contain the vertex. The closed link $L:=Q\cap\Sph^3$ is a Clifford torus up to isometry. The actual lifted link $\tCone\cap\Sph^3=L\setminus C_z$ is an open cylinder.
\end{proposition}

\begin{proof}
The lift condition is
\[
 z^2\overline{w^2}=(z\overline w)^2\in\R_{\geq0},
\]
which is equivalent to $z\overline w\in\R$, with $w\neq0$ retained from the base. Allowing $w=0$ gives precisely the closure $Q$. The points of $L$ have the form
\[
 \widetilde f(s,t)=(\cos s\,e^{it},\sin s\,e^{it}),
 \qquad s,t\in\R/2\pi\Z.
\]
The full parameter torus covers $L$ twice, with the identification $(s,t)\sim(s+\pi,t+\pi)$.

For an explicit isometry, put
\[
 U=\frac1{\sqrt2}
 \begin{pmatrix}1&-i\\-i&1\end{pmatrix}\in SU(2),
 \qquad
 (\zeta,\omega)=U(z,w)
 =\left(\frac{z-iw}{\sqrt2},\frac{-iz+w}{\sqrt2}\right).
\]
Then
\[
 |\zeta|^2-|\omega|^2
 =\frac12\bigl(|z-iw|^2-|w-iz|^2\bigr)
 =-2\Im(z\overline w).
\]
Thus
\[
 U(L)=T_{\mathrm{can}}
 :=\{|\zeta|=|\omega|=1/\sqrt2\},
 \qquad U(Q)=C(T_{\mathrm{can}}).
\]
The omitted circle $C_z$ maps to
$\{\omega=-i\zeta\}\cap T_{\mathrm{can}}$, an essential circle of slope $(1,1)$. It is nonseparating, so its complement is an open cylinder.

Equivalently, the Hopf map
\[
 \mathfrak H(z,w)=
 \bigl(2\Re(z\overline w),2\Im(z\overline w),|z|^2-|w|^2\bigr)
 =:(y_1,y_2,y_3)
\]
satisfies $L=\mathfrak H^{-1}(\Sph^2\cap\{y_2=0\})$. Here it is the inverse image of a great circle that is a torus; individual Hopf fibres are circles. This identifies the closed link, not the punctured lifted link.
\end{proof}

\section{Ros's quotient-helicoid theorem}\label{appendix:Ros-theorem}
\subsection{The area-minimising quotient helicoid}\label{subsec:quotient-helicoid}

We recall Ros's proof of \cite[Theorem~13, pp.~82--84]{Ros2006}, with the homological and limiting comparisons made explicit. In this appendix $\ell>0$ denotes the translation length, not the pitch. Set
\[
 N^{\ell}:=\R^3/\langle\tau^{\ell}\rangle,\qquad
 \tau^{\ell}(x_1,x_2,x_3)=(x_1+\ell,x_2,x_3),
\]
and let $\mathscr H_{\ell}\subset N^{\ell}$ be the image of
\[
 X_{\ell}(s,t)=\left(\frac{\ell}{\pi}t,s\cos t,s\sin t\right),
 \qquad s,t\in\R.
\]
The relation $X_{\ell}(-s,t+\pi)=\tau^{\ell}X_{\ell}(s,t)$ makes this a
properly embedded nonorientable minimal surface, topologically an open
M\"obius band. After a permutation of coordinates it is
$\overline H_{\ell/\pi}$ from Section~\ref{sec:helicoid}; thus
$\ell=\pi a$ in the pitch notation used there.

We use two results of Ros as black-boxes: the uniform curvature estimate for complete one-sided stable minimal surfaces in flat three-manifolds of uniformly bounded geometry, and the classification of complete noncompact nonorientable stable minimal surfaces in a rank-one translation quotient as quotient helicoids of total curvature $-2\pi$; see \cite[Corollary~11 and Theorem~12]{Ros2006}. Primitivity of the period will be deduced from this total curvature condition, not from embeddedness alone.

\begin{theorem}[Ros]\label{thm:quotient-helicoid-minimizing}
For every $\ell>0$, the chain $\bracket{\mathscr H_\ell}$ is locally mass-minimising modulo two in $N^\ell$. In particular, every smooth compact subdomain of $\mathscr H_\ell$ minimises area among compact orientable or nonorientable surfaces with the same boundary.
\end{theorem}

\begin{proof}
A homothety reduces the assertion to $\ell=1$. Write
$N=N^1=\R^3/\langle\tau\rangle$, where $\tau(x)=x+e_1$.
For $A,B\geq1$, let $G_{A,B}$ be generated by
\[
 \begin{gathered}
 v_1(x)=x+(1,0,0),\qquad v_2(x)=x+(0,0,A),\\
 g(x_1,x_2,x_3)=(-x_1,x_2+B,x_3).
 \end{gathered}
\]
The map $g$ is a glide reflection (see \cite{Ros2006} for the terminology). The quotient $M_{A,B}=\R^3/G_{A,B}$ is the product $K_{x_1,x_2}\times S^1_{x_3}$ of a flat Klein bottle with a circle. Its translation lattice is generated by $(1,0,0)$, $(0,2B,0)$ and $(0,0,A)$ and has index two in $G_{A,B}$.

\smallskip\noindent
K\"unneth's formula gives $H_2(M_{A,B};\Z_2)\simeq(\Z_2)^3$. Its coordinate classes are
\[
 \alpha_1=[x_3=0],\qquad \alpha_2=[x_2=0],\qquad \alpha_3=[x_1=0].
\]
They are represented respectively by a Klein bottle and two tori. Suppose an embedded closed union of totally geodesic surfaces represents some class. Its lifts are disjoint affine planes in $\R^3$, so all these planes are parallel: two nonparallel planes would intersect. For a normal $(n_1,n_2,n_3)$, invariance under the glide reflection requires $(-n_1,n_2,n_3)$ to be parallel to it. Thus either $n_1=0$ or $n_2=n_3=0$.

In the first case every plane contains the $x_1$-direction; the union comes from curves in the $(x_2,x_3)$-torus and its class lies in $\operatorname{span}\{\alpha_1,\alpha_2\}$. In the second case the planes are $x_1=\mathrm{constant}$, and their classes lie in the $H_1(K;\Z_2)\otimes H_1(S^1;\Z_2)$ summand, $\operatorname{span}\{\alpha_2,\alpha_3\}$. The two types cannot occur in the same disjoint union, again because their lifts would intersect. Therefore the class $\alpha:=\alpha_1+\alpha_3$ has no totally geodesic representative. This is the planar-union argument in Ros's construction.

Choose a mass-minimising cycle in the homology class $\alpha$. Existence and codimension-one regularity give a smooth embedded closed minimal surface, possibly disconnected. It has a nonflat component; keep one and denote it by $\Sigma_{A,B}$. This component still minimises mass in its own mod-two homology class: replacing it by a homologous chain of smaller mass and retaining the other components would decrease the mass of the original cycle, since any cancellations only decrease mass. In particular it is stable. We do \emph{not} assert minimisation against
all cycles in the compact manifold $M_{A,B}$.

The component is one-sided. Otherwise the constant function $1$ would be an admissible normal variation, and stability in the flat ambient manifold would then imply $0\leq-\int_{\Sigma_{A,B}}|A_{\Sigma_{A,B}}|^2$, contrary to being non-flat.

\smallskip\noindent
Set $c_{A,B}^2=\max_{\Sigma_{A,B}}|K_{\Sigma_{A,B}}|$. First lift
$\Sigma_{A,B}$ to the compact flat-torus cover associated with the translation lattice, keep a nonflat component, and then take its orientable cover if needed. Relative to the parallel frame on the flat torus, the Gauss map of this compact minimal surface is a nonconstant meromorphic map to $S^2$, hence is surjective. Descending a point with normal parallel to $e_1$ gives $q_{A,B}\in\Sigma_{A,B}$ with that normal direction. 

The normal exponential map is injective on the dics in the normal bundle of radius $r<c_{A,B}^{-1}$. To see this, it is initially injective and has no focal points in that range. With shape operator $-d\nu$, a parallel sheet at distance $r$ has principal curvatures
\[
 \frac{\lambda}{1-r\lambda},\qquad
 \frac{-\lambda}{1+r\lambda},
\]
so its scalar mean curvature (half the trace) relative to the outward normal of the tube is
\begin{equation}\label{eq:parallel-mean-curvature}
 H_r=\frac{r\lambda^2}{1-r^2\lambda^2}\geq0.
\end{equation}
At a first self-contact before the focal radius, the two local parallel sheets have facing outward normals. Write them as graphs $u\leq v$ over the common tangent plane, choosing the upward normal to agree with the outward normal of the lower sheet. Their mean-curvature inequalities are then opposite: the graph mean-curvature operator satisfies $\mathcal M(u)\geq0\geq\mathcal M(v)$. The strong maximum principle for $v-u\geq0$ forces coincidence on a neighbourhood. The two inequalities are then equalities, so \eqref{eq:parallel-mean-curvature} gives zero curvature on an open patch of the original surface. Unique continuation
contradicts nonflatness. This also covers contact at a point where the curvature was initially zero.

The normal geodesic at $q_{A,B}$ is parallel to the period-one direction $e_1$. Injectivity on the interval $(-c_{A,B}^{-1},c_{A,B}^{-1})$ gives
\begin{equation}\label{eq:curvature-lower-helicoid}
 \frac2{c_{A,B}}\leq1,\qquad c_{A,B}\geq2.
\end{equation}
The flat manifolds $M_{A,B}$ have a uniform positive injectivity-radius bound for $A,B\geq1$. Ros's curvature estimate (as recalled before the statement of the theorem) therefore gives
\begin{equation}\label{eq:curvature-upper-helicoid}
 c_{A,B}\leq C_0
\end{equation}
with $C_0$ independent of $A,B$.

\smallskip\noindent
Choose $A_j,B_j\to\infty$ and a base point $p_j\in\Sigma_j:=\Sigma_{A_j,B_j}$ where $|K_{\Sigma_j}|=c_{A_j,B_j}^2$. The pointed ambient manifolds converge to $(N,0)$. Fix $r_0<1/(2C_0)$. The embedded normal tube of radius
$r_0$ has Jacobian $1-r^2\lambda^2\geq3/4$. Its part over $\Sigma_j\cap B_L(p_j)$ lies in $B_{L+r_0}(p_j)$, hence
\begin{equation}\label{eq:tube-area-bound}
 \frac32r_0\,\mathcal H^2(\Sigma_j\cap B_L(p_j))
 \leq\operatorname{Vol}_{M_{A_j,B_j}}(B_{L+r_0}(p_j)).
\end{equation}
This supplies uniform local bounds. Curvature estimates then give smooth convergence on compact sets to a properly embedded limit, possibly disconnected. The convergence has multiplicity one: two sheets approaching the same limit sheet would have intersecting normal tubes of radius $r_0$.

Keep the connected component $\Sigma_\infty$ through $0$. It is proper and complete, and \eqref{eq:curvature-lower-helicoid} makes it nonflat. Furthermore, it cannot be compact. 
The limit inherits minimisation among compactly supported boundaries $\partial V$. To make the comparison explicit, suppose such a perturbation strictly decreases mass. Choose a smooth compact domain $\Omega\subset\Sigma_\infty$ whose interior contains $\Sigma_\infty\cap\spt(\partial V)$, and put $E=\bracket{\Omega}$. Then,
\[
 \M(E+\partial V)<\M(E).
\]
Identify a fixed ambient neighbourhood of $\spt V\cup\Omega$ with its image in $M_{A_j,B_j}$ for large $j$. Smooth multiplicity-one convergence gives nearby domains $\Omega_j\subset\Sigma_j$ and homotopies from $\Omega$ to $\Omega_j$. If $E_j=\bracket{\Omega_j}$, the homotopy formula gives swept chains $G_j$ of dimension three and $B_j$ of dimension two with
\[
 \partial G_j=E_j+E+B_j,\qquad
 \partial B_j=\partial E_j+\partial E,\qquad
 \M(B_j)\rightarrow0.
\]
Here $B_j$ is the swept boundary of $\Omega$, so the last assertion follows
from $C^1$ convergence. Also $\M(E_j)\to\M(E)$. The homologous competitor
\[
 \bracket{\Sigma_j}+\partial(V+G_j)
 =\bracket{\Sigma_j}+E_j+(E+\partial V)+B_j
\]
has mass at most
\[
 \M(\bracket{\Sigma_j})-\M(E_j)+\M(E+\partial V)+\M(B_j)
 <\M(\bracket{\Sigma_j})
\]
for large $j$, a contradiction. This proves the comparison for the selected component, without assuming that the entire pointed limit is connected.

Now $H_2(N;\Z_2)=H_2(S^1\times\R^2;\Z_2)=0$, so every compactly supported finite-mass rectifiable $2$-cycle in $N$ bounds a compactly supported finite-mass $3$-chain. Hence the comparison extends to all compact cycles, and $\bracket{\Sigma_\infty}$ is locally mass-minimising in the sense of Section~\ref{subsec:flat-chains}.

For completeness, two-sidedness can be excluded directly. The tube bound passes to the proper limit; since $N=S^1\times\R^2$ has quadratic volume growth, it gives $\mathcal H^2(\Sigma_\infty\cap\{|(x_2,x_3)|\leq R\})\leq C_1R^2$ for $R\geq1$. If $\Sigma_\infty$ were two-sided, use the logarithmic cutoff equal to $1$ for $|(x_2,x_3)|\leq R$, equal to $0$ for $|(x_2,x_3)|\geq R^2$, and linear in $\log|(x_2,x_3)|$ in between. Properness makes its support compact, and the area bound gives $\int|\nabla\chi_R|^2\leq C_2/\log R$. Stability yields
\[
 \int_{\Sigma_\infty}|A|^2\chi_R^2
 \leq\int_{\Sigma_\infty}|\nabla\chi_R|^2\longrightarrow0.
\]
Thus $A\equiv0$, a contradiction. Since $N$ is orientable, the one-sided surface $\Sigma_\infty$ is nonorientable.

\smallskip\noindent
Ros's rank-one stable classification now identifies $\Sigma_\infty$ as a quotient helicoid of total curvature $-2\pi$. Write its pitch as $b>0$. Invariance under the ambient period $1$ gives $1=m\pi b$ for a positive integer $m$; nonorientability makes $m$ odd. A fundamental quotient domain has total curvature
\[
 -\int_0^{m\pi}\int_\R
 \frac{b^2}{(b^2+s^2)^{3/2}}\dd s\dd t=-2\pi m.
\]
Thus $m=1$, $b=1/\pi$, and $\Sigma_\infty$ is congruent to
$\mathscr H_1$. Embeddedness alone would not give this conclusion: quotienting a helicoid by an odd multiple of its half-turn translation still gives an embedded nonorientable surface.

It follows that $\bracket{\mathscr H_1}$ is locally mass-minimising. Given a smooth compact subdomain $\Omega\subset\mathscr H_1$ and a compactly supported finite-mass chain $S$ with $\partial S=\partial\bracket{\Omega}$, set $Q=S+\bracket{\Omega}$ and choose $W\Subset N$ containing both supports. The local comparison and subadditivity give
\[
 \M(\bracket{\mathscr H_1}\mres W)
 \leq\M((\bracket{\mathscr H_1}+Q)\mres W)
 \leq\M(\bracket{\mathscr H_1\setminus\Omega}\mres W)+\M(S).
\]
Cancelling the common part proves $\mathcal H^2(\Omega)\leq\M(S)$. Finally, the homothety $x\mapsto\ell x$ sends $N^1$ to $N^\ell$ and $\mathscr H_1$ to $\mathscr H_\ell$, multiplying areas by $\ell^2$.
\end{proof}

\bibliographystyle{amsalpha}
\bibliography{arxiv}

\end{document}